\documentclass[11pt]{amsart}

\usepackage[utf8]{inputenc}
\usepackage[english]{babel}
\usepackage{graphicx}
\usepackage[colorlinks]{hyperref}
\usepackage[dvipsnames]{xcolor}

\usepackage{geometry}
\usepackage[backgroundcolor=white,linecolor=red,bordercolor=red]{todonotes}
\usepackage[shortlabels]{enumitem}
\setenumerate{label=(\roman*)}
\usepackage{hyperref}
\newcommand\myshade{100}
\hypersetup{
  linkcolor  = RoyalBlue!\myshade!black,
  citecolor  = ForestGreen!\myshade!black,
  urlcolor   = OrangeRed!\myshade!black,
  colorlinks = true,
}

\usepackage{amsmath}
\usepackage{amssymb}
\usepackage{amsfonts}
\usepackage{amsthm}
\usepackage{mathtools}
\usepackage{bm}
\usepackage{bbold}
\usepackage{tikz-cd}

\usepackage[capitalise]{cleveref}
\crefformat{equation}{(#2#1#3)}
\crefname{subsection}{Section}{Sections}

\theoremstyle{plain}
\newtheorem{thm}{Theorem}[section]
\newtheorem{lem}[thm]{Lemma}
\newtheorem{prop}[thm]{Proposition}
\newtheorem{cor}[thm]{Corollary}

\theoremstyle{definition}
\newtheorem{defn}[thm]{Definition}
\newtheorem{example}[thm]{Example}

\theoremstyle{remark}
\newtheorem*{rem}{Remark}

\makeatletter
\def\@fnsymbol#1{\ensuremath{\ifcase#1\or \dagger\or \ddagger\or
           \dagger\dagger
           \or \ddagger\ddagger \else\@ctrerr\fi}}
\makeatother

\newcommand\N{\ensuremath{\mathbb{N}}}
\newcommand\R{\ensuremath{\mathbb{R}}}
\newcommand\Z{\ensuremath{\mathbb{Z}}}

\newcommand\C{\ensuremath{\mathbb{C}}}

\newcommand{\el}[1]{\ensuremath{\ell_{#1}}}

\DeclareMathOperator{\supp}{supp}

\usepackage{csquotes}

\newcommand{\falg}{$f\!$-algebra}

\usepackage{thmtools}
\usepackage{thm-restate}

\begin{document}

\title{Banach lattice AM-algebras}

\author[D. Mu\~noz-Lahoz]{David Mu\~noz-Lahoz}
\address{Instituto de Ciencias Matem\'aticas (CSIC-UAM-UC3M-UCM)\\Consejo Superior de Investigaciones Cient\'ificas\\C/ Nicol\'as Cabrera, 13--15, Campus de Cantoblanco UAM\\28049 Madrid, Spain.}
\email{david.munnozl@uam.es}

\author[P. Tradacete]{Pedro Tradacete}
\address{Instituto de Ciencias Matem\'aticas (CSIC-UAM-UC3M-UCM)\\Consejo Superior de Investigaciones Cient\'ificas\\C/ Nicol\'as Cabrera, 13--15, Campus de Cantoblanco UAM\\28049 Madrid, Spain.}
\email{pedro.tradacete@icmat.es}

\subjclass[2020]{46B42, 46J10, 46J30, 06F25} 
%46B42 - Banach lattices; 
%06F25 - Ordered rings, algebras, modules; 
%46J10   	Banach algebras of continuous functions, function algebras
%46J30   	Subalgebras of commutative topological algebras

\keywords{Banach lattice algebra; AM-space; spaces of continuous functions}

%\date{\today}

\begin{abstract}
    An analogue of Kakutani's representation theorem for Banach lattice algebras is provided. We characterize Banach lattice algebras that embed as a closed sublattice-algebra of $C(K)$ precisely as those with a positive approximate identity $(e_\gamma )$ such that $x^{*}(e_\gamma )\to \|x^{*}\|$ for every positive functional $x^{*}$. We also show that every Banach lattice algebra with identity other than $C(K)$ admits different product operations which are compatible with the order and the algebraic identity. This complements the classical result, due to Martignon, that on $C(K)$ spaces pointwise multiplication is the unique compatible product.
\end{abstract}

\maketitle

\section{Introduction}\label{sec:intro}

The space $C(K)$ of continuous functions on a compact Hausdorff space $K$, with its usual uniform norm $\|\cdot\|_\infty$, pointwise order and pointwise operations can be seen both as a Banach lattice and as a Banach algebra. In fact, its lattice and algebra structures are very nicely interwoven: for instance, the notions of closed lattice ideals and closed algebraic ideals coincide (and are in correspondence with the closed subsets of $K$ as zero sets); also, it is well known that closed subalgebras of $C(K)$ are sublattices, and closed sublattices containing the constant functions are necessarily subalgebras of $C(K)$ (the latter fact playing a role in some proofs of the classical Stone--Weierstrass theorem). Finally, one can easily observe that the product of positive functions in $C(K)$ is again positive. This seemingly na\"ive property relating the order and algebra structures lies behind the definition of a Banach lattice algebra: a \textit{Banach lattice algebra} is a Banach lattice that is also a Banach algebra in which the product of positive elements is positive. In the last few years, there has been an increasing interest in understanding the structure of Banach lattice algebras (\cite{alekhno2012,alekhno2018,blanco2023,de_jeu2021,jaber2020,wickstead2017_questions,wickstead2017_ordered,wickstead2017_2d}), thus reviving earlier developments initiated in the 1980s and 90s (\cite{grobler1988,huijsmans1988,huijsmans1995,scheffold1980,scheffold1984,scheffold1988}).

From the point of view of Banach lattices, $C(K)$ spaces are very well understood. Recall that a Banach lattice $X$ is said to be an AM-space if
$\|x\vee y\|=\|x\|\vee \|y\|$ for all $x,y \in X_+$. If there exists a
strong unit $e \in X_+$ such that $\|x\|=\inf \{\, \lambda >0 : |x|\le
\lambda e\, \} $ for all $x \in X$, then $X$ is called an
AM-space with unit $e$. S.\ Kakutani proved that AM-spaces are
just the closed sublattices of $C(K)$:

\begin{thm}[\cite{kakutani1941}]\label{thm:kakutani}
    \begin{enumerate}
        \item[]
        \item Every AM-space with unit is lattice isometric to $C(K)$
            for some compact Hausdorff space $K$, with the unit
            corresponding to the constant one function $\mathbb{1}_K$.
        \item Every AM-space is lattice isometric to a closed
            sublattice of $C(K)$ for some compact Hausdorff space $K$.
            More precisely, there exists a family of pairs of points
            $\{(t_i,s_i)\}_{i \in I}\subseteq K\times K$ and scalars
            $\{\lambda _i\}_{i \in I}\subseteq [0,1)$ such that the
            AM-space is lattice isometric to the closed sublattice of
            $C(K)$:
            \[
            \{\, f \in C(K) : f(t_i)=\lambda _i f(s_i)\text{ for all
            }i \in I \, \}.
            \]
    \end{enumerate}
\end{thm}
The non-trivial implication in Kakutani's theorem is particularly relevant in Banach lattice theory, as it allows us to view elements in abstract Banach lattices as true functions on some compact Hausdorff space (see for instance \cite[Section 1.b]{lindenstrauss_tzafriri1979}).

From a Banach lattice algebra point of view, it is natural to wonder whether an analogue intrinsic characterization might exist for closed sublattices of $C(K)$ that are also subalgebras. The aim of \cref{sec:AMalg} is to answer this question. More
precisely, we define an \textit{AM-algebra with unit} to be a Banach lattice
algebra which is an AM-space with unit $e$, where $e$ is also an algebraic
identity. We also define an \textit{AM-algebra with approximate unit} to be a Banach lattice
algebra with an approximate identity $(e_\gamma )$ such that
$x^{*}(e_\gamma )\to \|x^{*}\|$ for every positive functional $x^{*}$.
After a discussion of these notions, we prove the following
analogue of Kakutani's representation theorem for Banach lattice
algebras.

\begin{restatable}{thm}{generalAMalg}
    \label{thm:generalAMalg}
    \begin{enumerate}
         \item[]
         \item Every AM-algebra with unit is lattice and algebra
            isometric to $C(K)$ for some compact Hausdorff space $K$,
            with the unit corresponding to the constant one function $\mathbb{1}_K$.
        \item Every AM-algebra with approximate unit is lattice and
            algebra isometric to a closed sublattice-algebra of
            $C(K)$ for some compact Hausdorff space $K$. More
            precisely, there exists a closed set $F\subseteq K$ such that the
            AM-algebra is lattice and algebra isometric to the closed
            sublattice-algebra of $C(K)$:
            \[
            \{\, f \in C(K) : f(t)=0\text{
            for all }t \in F\, \} .
            \]
            In particular, it embeds as an order and algebraic ideal in $C(K)$.
    \end{enumerate}
\end{restatable}

Let us explain further the analogy between Theorems~\ref{thm:kakutani} and~\ref{thm:AMalg}. By Kakutani's theorem, every AM-space with unit is lattice isometric to a $C(K)$ space. When we also have a Banach lattice algebra structure, it seems that the most natural compatibility condition between the order unit and the algebraic structure is to ask for the order unit to be also an algebraic identity. According to previous theorem, this condition characterizes $C(K)$ spaces up to lattice and algebra isometry, just like Kakutani's theorem characterizes $C(K)$ spaces up to lattice isometry as those being an AM-space with unit.

Similarly, by Kakutani's theorem, every AM-space is lattice isometric to a closed sublattice of $C(K)$. Being an AM-space is equivalent to having a positive net $(e_\gamma)$ such that $x^*(e_\gamma)\to \|x^*\|$ for every positive functional $x^*$ (see \cref{lem:approx_AM}). Such a net is what we call an approximate order unit, establishing an analogy with the fact that the norm of every positive functional is attained at an order unit (see \cref{lem:def_approx_unit}). When we also have a Banach lattice algebra structure, it seems that the most natural compatibility condition between the AM-space structure and the algebraic structure is to ask for such an approximate order unit to be an approximate algebraic identity. This property characterizes the closed sublattice-algebras of $C(K)$ up to isometry, just like Kakutani's theorem characterizes closed sublattices of $C(K)$, up to isometry, as those having an approximate order unit (i.e., as those being AM-spaces).

In our proof of Theorem \ref{thm:generalAMalg}, the following result, due to L.\ Martignon, plays a key role.

\begin{prop}[{\cite[Proposition 1.4]{martignon1980}}]\label{prop:martignon}
    Let $K$ be a
    compact Hausdorff space and let $\star\colon C(K)\times C(K)\to
    C(K)$ be a binary operation such that $C(K)$ with
    pointwise order and product $\star$ is a vector lattice algebra.
    If in addition $\mathbb{1}_K$ is the identity of
    $\star$, then $\star$ is the pointwise multiplication.
\end{prop}

In other words, if we fix the lattice structure to be given pointwise and
the algebraic identity to be given by the constant one function $\mathbb{1}_K$, there
exists a unique vector lattice algebra structure on $C(K)$. Since this fact is fundamental in our approach, one could
wonder whether there might be other Banach lattice algebras with the same property.
In \cref{sec:products}, it will be shown that this is not the case. More precisely, we
have the following.

\begin{restatable}{thm}{products}
    \label{thm:products}
    Let $A$ be a Banach lattice algebra with identity $e$ and order $\le $. Suppose that
    there is a unique product on $A$ that makes it into a vector
    lattice algebra with identity $e$ and order $\le $.
    Then $A$ is lattice and algebra isometric to
    $C(K)$, with $e$ corresponding to $\mathbb{1}_K$.
\end{restatable}

% It is a remarkable remarkable result by J.\ Ma and W.\ J.\
% Wojciechowski \cite{ma_woj2002} that, for every $n \in \N$, the space $M_n(\R)$ of $n\times n$ real
% matrices enjoys the ``complementary'' property,
% in the sense that fixed the algebraic structure to be the usual matrix
% multiplication (i.e., the composition of operators) there is a unique
% lattice order that makes it into a vector lattice algebra, up to
% isomorphism.

% Theorem2.58 Aliprantis Burkinshaw

% \begin{question}\todo{Solve these!}
%     In spaces of matrices, the only lattice order is entrywise.
%     \begin{enumerate}
%         \item What happens in $\Lr X$, with $X$ order complete? Is
%             there more than one lattice order?
%         \item Is this the unique space with this property (up to
%             isomorphism)?
%     \end{enumerate}
% \end{question}

\section{Sublattice-algebras of $C(K)$}\label{sec:AMalg}

%\section{Discussion and proof of \cref{thm:generalAMalg}}\label{sec:AMalg}

A \emph{Banach lattice algebra} is a Banach lattice that
at the same time is a Banach algebra in which the product of positive
elements is positive. A subset which is both a sublattice and a subalgebra of a Banach lattice algebra will be called a \emph{sublattice-algebra}. We say that a Banach lattice algebra \emph{has an identity} if it has an algebraic identity of norm one (which is in fact positive, see \cite{bernau_huijsmans1990_unit,wickstead2017_questions}).
The space of continuous functions $C(K)$ on a compact Hausdorff space $K$, with pointwise order and product, is an example of a
Banach lattice algebra with identity. Further classical examples are those of the form $\ell_1(G)$, for a group $G$, with the convolution product and usual norm and lattice operations, or $L_r(E)$, the space of regular operators on a Dedekind complete Banach lattice $E$. Some examples of Banach
lattice algebras that need not have an identity are certain closed sublattice-algebras of $C(K)$. Our goal is to characterize those Banach lattice
algebras that are isometric to a closed sublattice-algebra of $C(K)$.

First note is that if $f\wedge g=0$ in $C(K)$,
and $h \in C(K)_+$, then $(hf)\wedge g=0=(fh)\wedge g$. Any
sublattice-algebra of $C(K)$ must also satisfy this
condition; that is, any closed sublattice-algebra of $C(K)$ must be a
\emph{Banach \falg}. This shows that not every Banach lattice algebra
which is also an AM-space must be a sublattice-algebra of some $C(K)$: the Banach lattice $\el
\infty ^2$, with product $(x_1,y_1)(x_2,y_2)=(0,x_1x_2)$, is an
$AM$-space, that is also a Banach lattice algebra, but cannot be a
sublattice-algebra of $C(K)$, since 
\[
    (1,0)^2\wedge (0,1)=(0,1)\text{ while }(1,0)\wedge (0,1)=(0,0),
\]
(i.e., it is not an \falg). Therefore, some relation between the AM-space
and algebraic structures has to be required. In the unital case, this relation is
straightforward.

\begin{defn}
    Let $A$ be a Banach lattice algebra with identity $e$. If $A$
    is also an AM-space with (order) unit $e$, we say that $A$ is an
    \emph{AM-algebra with unit $e$}.
\end{defn}

The following theorem asserts that AM-algebras with unit are nothing
but algebras of continuous functions, thus proving the first part of
\cref{thm:generalAMalg}. As observed in \cite{martignon1980}, this is
a direct consequence of \cref{thm:kakutani} and \cref{prop:martignon}.

\begin{thm}\label{thm:AMalg}
    Every AM-algebra with unit is algebra and
    lattice isometric to $C(K)$, for a certain compact Hausdorff space $K$.
\end{thm}

In order to prove the second part of \cref{thm:generalAMalg}, we need to
define the notions of AM-spaces and AM-algebras with approximate
unit. The following is a well-known property that motivates our
definition.

\begin{lem}\label{lem:def_approx_unit}
    A Banach lattice $X$ is an AM-space with unit $e \in X_+$ if and
    only if $x^{*}(e)=\|x^{*}\|$ for all $x^{*} \in (X^{*})_+$.
\end{lem}
\begin{proof}
    If $B_X$ denotes the unit ball of $X$ and $e \in X_+$ is an order
    unit, then $B_X=[-e,e]$. For every $x^{*} \in (X^{*})_+$:
    \[
        \|x^{*}\|=\sup_{x \in (B_X)_+} x^{*}(x)=\sup_{x \in [0,e]}
        x^{*}(x)=x^{*}(e).
    \]
    Conversely, if $x^{*}\in (X^{*})_+$ and $x \in (B_X)_+$, by assumption
    \[
    0\le x^{*}(x)\le \|x^{*}\|=x^{*}(e).
    \]
    Since this holds for every $x^{*}\in (X^{*})_+$, it follows that
    $0\le x\le e$. Since $x \in (B_X)_+$ was arbitrary,
    $(B_X)_+\subseteq [0,e]$. Also
    \[
    \|e\|=\sup_{x^{*} \in (B_{X^{*}})_+} x^{*}(e)=\sup_{x^{*} \in
    (B_{X^{*}})_+}\|x^{*}\|=1,
    \]
    so $[0,e]\subseteq (B_X)_+$. It follows that $B_X=[-e,e]$, and so
    $X$ is an AM-space with unit $e$.
\end{proof}

\begin{defn}
    Let $X$ be a Banach lattice, and let $(e_\gamma )\subseteq X_+$ be
    a net. We say that $X$ is an \emph{AM-space with approximate unit
    $(e_\gamma )$} if $x^{*}(e_\gamma )\to \|x^{*}\|$ for every
    $x^{*}\in (X^{*})_+$. We also say that $(e_\gamma )$ is an
    \emph{approximate order unit} of $X$.
\end{defn}

For this definition to make sense, we need to check that spaces with an
approximate order unit are indeed AM-spaces. In fact, this
definition is equivalent to that of an AM-space.

\begin{lem}\label{lem:approx_AM}
    A Banach lattice $X$ has an approximate order
    unit if and only if $X$ is an AM-space.
\end{lem}
\begin{proof}
    Let $(e_\gamma )\subseteq X_+$ be an approximate order unit.
    For $x^{*}, y^{*} \in (X^{*})_+$:
    \begin{align*}
        \|x^{*}+y^{*}\|&=\lim (x^{*}+y^{*})(e_\gamma )\\
                       &=\lim x^{*}(e_\gamma )+\lim
                       y^{*}(e_\gamma )\\
                       &=\|x^{*}\|+\|y^{*}\|.
    \end{align*}
    This proves that $X^{*}$ is an AL-space. Hence $X$ is an AM-space.

    Conversely, suppose that $X$ is an AM-space. The positive
    unit ball $(B_X)_+$ is an increasing net; denote it by $(e_\gamma
    )$. For $x^{*}\in (X^{*})_+$, $x^{*}(e_\gamma )$ is an
    increasing and bounded net of real numbers. As such, its
    limit coincides with its supremum:
    \[
    \lim_\gamma  x^{*}(e_\gamma )=\sup_\gamma  x^{*}(e_\gamma )=\sup_{x
    \in (B_X)_+} x^{*}(x)=\|x^{*}\|.
    \]
    This shows that $(B_X)_+$ is an approximate order unit.
\end{proof}

\begin{rem}
    A priori, we do not assume that approximate order units are bounded, see next example. However, if the approximate order unit is countable, then it must be bounded by the Uniform Boundedness Principle.

    Previous lemma also shows that, if a Banach lattice has an approximate order unit, then it has a bounded approximate order unit (namely, the positive unit ball). In the separable case, one can even find a countable bounded approximate order unit.
\end{rem}

\begin{example}
    The following is an example of an approximate order unit in
    $C[0,1]$ with no bounded tails. Denote by $M[0,1]$ the dual of $C[0,1]$, and by
    \[
        U(\mu_1,\ldots,\mu_n;\varepsilon)=\{\, f\in C[0,1] :
        |\mu_i(f)|<\varepsilon\text{ for }i=1,\ldots,n\,\},
    \]
    where $\mu _1,\ldots ,\mu _n \in M[0,1]$ and $\varepsilon >0$.
    These form a basis of neighborhoods at $0$ for the weak topology. Define
    \[
    \mathcal{U}=\{\, U(\mu _1,\ldots ,\mu _n;\varepsilon ) : \mu _1,\ldots
    ,\mu _n \in M[0,1],\; \varepsilon >0\, \}.
    \]
    For every $(U,n)\in \mathcal{U}\times \N$, we are going to construct a
    continuous function $f_{(U,n)} \in \mathcal{U}$ with
    $\|f_{(U,n)}\|_\infty =n$. Suppose $U=U(\mu _1,\ldots ,\mu _k;\varepsilon
    )$, and define $\mu = |\mu _1|+ \cdots +|\mu _k|$. Decompose $\mu =\mu
    _c+\mu _d$, where $\mu _c$ and $\mu _d$ are the continuous and
    discrete parts of $\mu $, respectively. The support of $\mu _d$ is at
    most countable, say $\supp(\mu _d)=\{p_i\}_{i=1}^{\infty }$. Choose
    $N \in \N$ such that $\mu _d(\{p_i\}_{i=N}^{\infty })< \varepsilon
    /(2n)$. Since $\mu _c$ is continuous, there exists a non-empty open set
    $V\subseteq [0,1]\setminus\{p_i\}_{i=1}^{N}$ such that $\mu
    _c(V)<\varepsilon /(2n)$. Fix an arbitrary point $x_0 \in V$. By
    Urysohn's Lemma, there exists a continuous function $f\colon [0,1]\to
    [0,1]$ such that $0\le f\le \mathbb{1}$, $\supp(f)\subseteq V$ and
    $f(x_0)=1$. Then $f_{(U,n)}=nf$ is such that $\|f_{(U,n)}\|_\infty =n$
    and
    \[
    |\mu _i(f_{(U,n)})|\le n\mu (f)\le n \mu (V)<\varepsilon\quad \text{for }i=1,\ldots,k
    \]
    so $f_{(U,n)} \in U$.
    
    If we order $\mathcal{U}\times \N$ by setting $(U,n)\le (V,m)$ if and
    only if $V\subseteq U$ and $n\le m$, then the net
    $\{f_{(U,n)}\}_{(U,n)\in \mathcal{U}\times \N}$ weakly converges to
    zero and is unbounded. Indeed, if $W$ is a weak neighbourhood of $0$,
    there exists $U_0 \in \mathcal{U}$ such that $U_0\subseteq W$, and then
    $f_{(U,n)} \in W$ for every $(U,n)\ge (U_0,1)$. Yet $f_{(U,n)}$ is
    unbounded as we increase $n$; in particular, it has no bounded tail. It follows that
    $\{\mathbb{1}+f_{(U,n)}\}_{(U,n)\in \mathcal{U}\times \N}$ is an
    approximate order unit with no bounded tail.
\end{example}

Of course, AM-spaces with unit are the same as AM-spaces with an
approximate order unit that is constant. The following is another
characterization of approximate order units that will be useful later.

\begin{lem}\label{lem:approx_id_bidual}
    Let $X$ be a Banach lattice. A net $(e_\gamma )\subseteq X_+$ is
    an approximate order unit if and only if $X^{**}$ is lattice isometric
    to $C(K)$, for a certain compact Hausdorff space $K$, and $e_\gamma
    \xrightarrow{w^{*}} \mathbb{1}_K$.
\end{lem}
\begin{proof}
    If $(e_\gamma )$ is an approximate order unit, $X$ is an AM-space by
    \cref{lem:approx_AM}, and therefore its bidual, being an AM-space
    with unit, is lattice isometric to $C(K)$ for a certain compact
    Hausdorff space $K$. Moreover, the element $\mathbb{1}_K$ is such
    that $\mathbb{1}_K(x^{*})=\|x^{*}\|$ for all $x^{*}\in (X^{*})_+$,
    because $B_{X^{**}}=[-\mathbb{1}_K,\mathbb{1}_K]$.
    Then $x^{*}(e_\gamma )\to \|x^{*}\|=\mathbb{1}_K(x^{*})$ for all
    $x^{*}\in (X^{*})_+$; in other words,
    $e_\gamma\to \mathbb{1}_K$ in the weak$^*$ topology. The converse follows
    easily from this last observation.
\end{proof}

The algebraic counterpart of approximate order units are approximate
algebraic identities. A net $(e_\gamma )$ in a Banach algebra $A$ is
said to be a \emph{left} (resp.\ \emph{right}) \emph{approximate
(algebraic) identity} if $e_\gamma x\to x$ (resp.\ $xe_\gamma \to x$)
for all $x \in A$. Of course, this is a standard definition in Banach
algebras (see \cite[Chapter 5]{palmer1994}). When we omit the left or right specification, it
means that it is both a left and a right approximate identity.

Putting together the notions of approximate order unit and
approximate algebraic identity,
we get the approximate analogue of AM-algebras with unit.

\begin{defn}
    Let $A$ be a Banach lattice algebra and let $(e_\gamma )\subseteq
    A_+$. We say that $A$ is an \emph{AM-algebra with approximate unit
    $(e_\gamma )$} if $(e_\gamma )$ is both an approximate order unit
    and an approximate algebraic identity.
\end{defn}

\begin{example}\label{ex:approxAMalg}
    \begin{enumerate}
        \item An AM-algebra with unit $e$ is also an AM-algebra with
            approximate unit, in which the approximate unit is the
            constantly $e$ sequence.
        \item The Banach lattice $c_0$ with supremum norm and
            coordinatewise order and product is an AM-algebra with
            approximate unit $\big(\sum_{i=1}^{n}e_i\big)_n$, where
            $e_i(i)=1$ and $e_i(j)=0$ for $i\neq j$.
        \item Any Banach lattice algebra with identity that is not an
            AM-space shows that we can have approximate algebraic
            identities without having approximate order units. This is the case for the Banach lattice algebra
            $\el 1(\Z)$, where the product is given by convolution.
        \item If $X$ is an AM-space, we can always endow it with the
            identically zero product, and it will become a Banach
            lattice algebra with an approximate order unit that has no
            approximate algebraic identity (unless $X=\{0\}$).
        \item Even if a Banach lattice algebra has an approximate order unit and
            an approximate algebraic
            identity,
            it may not be an AM-algebra with approximate unit, see
            \cref{ex:nonAMalg}.
    \end{enumerate}
\end{example}

AM-algebras with approximate unit are nothing but the closed
sublattices of $C(K)$ that are, at the same time, subalgebras.

\begin{thm}\label{thm:approxAMalg}
    A Banach lattice algebra $A$ is an AM-algebra with approximate
    unit if and only if it is lattice and algebra isometric to a closed
    sublattice-algebra of $C(K)$, for a certain compact
    Hausdorff $K$.
\end{thm}

For the proof of this theorem we need to introduce some properties of the
Arens products. Given Banach spaces $A$, $B$ and $C$, and a bounded bilinear map $P\colon A\times B\to C$, its
\emph{Arens adjoint} is the bounded bilinear map $P^{*}\colon
C^{*}\times A\to B^{*}$ defined by $P^{*}(\phi ,a)(b)=\phi (P(a,b))$
for every $a \in A$, $b \in B$ and $\phi \in C^{*}$; when $A=B=C$, its \emph{transpose} is the
bounded bilinear map $P^{t}\colon A\times A\to A$ defined by
$P^{t}(b,a)=P(a,b)$ for all $a,b \in A$. The bilinear map $P$ is
called \emph{Arens regular} if $P^{t * * * t}=P^{* * *}$.

In the particular case that $A$ is a Banach algebra with
product $P\colon A\times A\to A$, both $P^{* * *}$ and $P^{t * * * t}$
define products on $A^{* *}$, called the \emph{first} and \emph{second
Arens product}, respectively.
When $A$ is a Banach lattice algebra, the first and second Arens
product make $A^{* *}$ into a Banach lattice algebra with the usual
lattice structure (see \cite{huijsmans_depagter1984}). These products extend
$P$, in the sense that the canonical isometry
$A\to A^{* *}$ is an algebra homomorphism. When $P$ is Arens
regular, the first and second Arens product coincide, and the algebra
$A$ is said to be \emph{Arens regular}.

It turns out that identities for the Arens products are closely related to
approximate identities in the original algebra.

\begin{prop}[{\cite[Proposition 5.1.9]{palmer1994}}]\label{prop:unit_arens}
    Let $A$ be a Banach algebra, with canonical embedding $j\colon
    A\to A^{* *}$, and let $(e_\gamma )$ be a left
    (resp.\ right) approximate identity. If $j(e_\gamma)$ weak$^*$
    converges to some $e \in A^{* *}$, then $e$ is a left (resp.\
    right) identity for the second (resp.\ first) Arens product.
\end{prop}

Arens regularity of bilinear operators and algebras has been
extensively studied in the literature. In this line,
G.\ Buskes and R.\ Page provided in \cite{buskes_page2005} several equivalent conditions for all
bilinear and positive operators defined on a Banach lattice to be
Arens regular. Here we will only need the following one.

\begin{thm}[{\cite{buskes_page2005}}]\label{thm:arensreg}
    Let $E$ be a Banach lattice. Every positive bilinear map $P\colon
    E\times E\to E$ is Arens regular if and only if $\el 1$ does not lattice
    embed in $E$.
\end{thm}

We are now ready to prove \cref{thm:approxAMalg}.

\begin{proof}[Proof of \cref{thm:approxAMalg}]
    Let $A$ be an AM-algebra with approximate unit $(e_\gamma
    )\subseteq A_+$. Since $A$ is an AM-space, its bidual $A^{**}$ is
    lattice isometric to $C(K)$ for a certain compact Hausdorff $K$.
    From \cref{thm:arensreg} it follows that $A$ is Arens regular,
    because, being an AM-space, $\el 1$ does not embed, as a lattice, in $A$. Endow
    $A^{* *}$ with the Arens product. Then $A^{* *}$ is a Banach
    lattice algebra, and the canonical isometric embedding $j\colon
    A\to A^{* *}$ is both a lattice and algebra homomorphism.

    By \cref{lem:approx_id_bidual}, $j(e_\gamma
    )\xrightarrow{w^{*}}\mathbb{1}_K$ in $C(K)$. At the same time, by
    \cref{prop:unit_arens},
    we have that $\mathbb{1}_K$ is the identity of $A^{* *}$, because
    both Arens products coincide, and $\mathbb{1}_K$ is the $w^{*}$-limit of the
    two-sided approximate identity $(e_\gamma )$. Recapitulating, we have
    that $A^{* *}$ is lattice
    isometric to $C(K)$, and we have endowed this space with a
    product, the Arens product, that makes it into a Banach lattice
    algebra with identity $\mathbb{1}_K$. According to
    \cref{prop:martignon}, this product can only be the pointwise
    product, so $A^{* *}$ is both lattice and algebra isometric to
    $C(K)$.

    Conversely, if $A$ is a closed sublattice-algebra of
    $C(K)$, then the positive unit ball $(B_A)_+$ is an increasing net
    that is an approximate order unit by \cref{lem:approx_AM}. It is
    also an approximate
    algebraic identity, by the standard argument of C*-algebra theory
    (see \cite[Theorem I.4.8]{davidson1996} and note that, even though the
    natural scalar field for C*-algebras is \C, the same argument works for closed subalgebras of a real $C(K)$). Hence $A$ is an AM-algebra with approximate unit.
\end{proof}

At this point, the proof of \cref{thm:generalAMalg} is almost done.
The remaining details are completed after the restatement of the
theorem.

\generalAMalg*
\begin{proof}
    $(i)$ is contained in \cref{thm:AMalg}. To prove $(ii)$, let $A$ be
    an AM-algebra with approximate unit. According to
    \cref{thm:approxAMalg}, we can see $A$ as a closed sublattice-algebra of $C(L)$, for a certain compact Hausdorff $L$. Since
    $A$ is a closed sublattice of $C(L)$, in virtue of another result
    by S.\ Kakutani (see \cite[Theorem 3]{kakutani1941}), there exists a family of
    pairs of points $\{(t_i,s_i)\}_{i \in I}\subseteq L\times L$ and
    scalars $\{\lambda _i\}_{i \in I} \subseteq [0,1]$ such that
    \[
    A=\{\, f\in C(L) : f(t_i)=\lambda _if(s_i)\text{ for all }i \in I
    \, \}.
    \]
    Given $i \in I$, if $f \in A$, then $f^2(t_i)=f(t_i)^2=\lambda
    _i^2f^2(s_i)$, but since also $f ^2 \in A$, $f^2(t_i)=\lambda _i
    f^2(s_i)$, and we must have $\lambda _i^2f(s_i)^2=\lambda _if(s_i)^2$.
    This only leaves three possibilities: either $f(s_i)=0$ for
    every $f \in A$, in which case $f(t_i)=0=f(s_i)$, or $\lambda _i=0$,
    in which case $f(t_i)=0$, or $\lambda _i=1$, in which case
    $f(t_i)=f(s_i)$.

    Collecting these cases, one gets a family of pairs of points $\{(u_j,v_j)\}_{j \in J}\subseteq L\times L$ and a family of points $\{w_p\}_{p \in P}\subseteq L$ such that
    \[
    A=\{\, f\in C(L) : f(u_j)=f(v_j)\text{ and }f(w_p)=0\text{ for all }j \in J, p \in P\, \}.
    \]
    To obtain the desired result, it only remains to ``glue together'' $u_j$ and $v_j$ for every $j \in J$. Even though this is a standard procedure, we work out the details for the sake of completeness.
    
    Consider the relation defined by $\{(u_j,v_j)\}_{j \in J}$ in $L$. Equivalently, $t\sim s$ in $L$ if and only if $f(t)=f(s)$ for all $f\in A$. This is certainly an equivalence relation. Consider the set $K=L/\sim$ with the quotient topology, and let $Q\colon L\to K$ be the quotient map. Note that for each $f \in A$, the map $\tilde f\colon K\to \R$ defined by $\tilde f(Qs)=f(s)$ is well-defined and continuous. Whenever $Qs\neq Qt$, there exists $f\in A$ such that $f(s)\neq f(t)$. Hence, $\tilde f(Qs)\neq \tilde f(Qt)$, which in particular implies that $K$ is Hausdorff. Since $K=Q(L)$, and $Q$ is continuous, it also follows that $K$ is compact.

    There is an isometric embedding
    \[
    \begin{array}{cccc}
    T\colon& A & \longrightarrow & C(K) \\
            & f & \longmapsto & \tilde f \\
    \end{array}.
    \]
    Indeed, $T$ is certainly linear and
    \[
    \sup_{Q(s) \in K}|\tilde f(Q(s))|=\sup_{s\in L}|f(s)|.
    \]
    Also,
    \[
    |\tilde f|(Qs)=|\tilde f (Qs)|=|f(s)|=|f|(s)\text{ and }\widetilde{fg}(Qs)=(fg)(s)=f(s)g(s)=\tilde f(Qs) \tilde g(Qs),
    \]
    for all $f,g \in A$ and $s \in L$, so $T$ is a lattice and algebra homomorphism. Let now $F$ be the closure of $\{Qw_p:p\in P\}\subset K$. It is straightforward to check that
    \[
    T(A)=\{\, g\in C(K) : g(t)=0\text{ for all }t \in F\, \}.\qedhere
    \]
\end{proof}

\begin{rem}
Previous result could also be compared to other works which study embeddings of an (ordered) Banach algebra as an (ordered) subalgebra of $C(K)$ such as \cite{Render92}. In \cite[Corollary 3.5]{Render92}, H.\ Render shows that a Banach algebra with a bounded approximate identity and a closed multiplicative cone containing all squares is order and algebra isomorphic to a subalgebra of $C(K)$. Note that this result is quite different from the one presented here, since H.\ Render is not assuming that the algebra is a lattice, and the conclusion is not of an isometric nature.

Moreover, the condition that squares are positive is closely related to being an \emph{almost \falg} (i.e., $a\wedge b=0$ implies $ab=0$). This connects with the result by E.\ Scheffold \cite[Satz 2.4]{scheffold1981} that an almost \falg\ is semi-simple (i.e., its Jacobson radical is zero) if and only if it is isomorphic to a separating sublattice-algebra of $C_0(X)$, where $X$ is a locally compact Hausdorff space. Note that this result characterizes separating sublattice-algebras of $C_0(X)$, but only in an isomorphic manner. The theorem given here, however, characterizes sublattice-algebras of $C(K)$ isometrically. Both Render and Scheffold's results use adaptations of the Gelfand transform, while the technique presented here relies heavily on Banach lattice results and \cref{prop:martignon}.
\end{rem}

To finish this section, we extract a couple of consequences of the
theorem. The first one is straightforward.

\begin{cor}
    Every AM-algebra with approximate unit is an \falg.
\end{cor}

With this corollary we can come back to the example left hanging in
\cref{ex:approxAMalg}.

\begin{example}\label{ex:nonAMalg}
    Even if a Banach lattice algebra has an approximate order unit and
    an approximate algebraic identity, it need not be an AM-algebra with
    approximate unit. Consider the Banach
    lattice $c_0\oplus \R$ with supremum norm and coordinatewise
    order. For $(x,\lambda ),(y,\mu )\in c_0\oplus \R$, define
    the product:
    \[
        (x,\lambda )(y,\mu )=(xy+\lambda y+\mu x,\lambda \mu
        ).
    \]
    Then $c_0\oplus \R$ is a Banach lattice algebra with identity
    $(0,1)$. It is also an AM-space, so it has an approximate
    order unit. However, it is not an AM-algebra with
    approximate unit, because it is not an \falg\ (if $x\in c_0$ is nonzero, then $(x,0)\wedge (0,1)=(0,0)$, but $(x,0)(0,1)=(x,0)\neq (0,0)$).
    % Indeed, if $(e_\gamma )\subseteq
    % (c_0\oplus \R)_+$ were an approximate order unit and an
    % approximate algebraic
    % identity, we have shown that $A^{* *}$ with the Arens product is just
    % $C(K)$ for some compact Hausdorff space $K$. Since $j(0,1)$, where
    % $j\colon A\to A^{* *}$ denotes the canonical map, is the
    % identity of the Arens product, it must be $j(0,1)=\mathbb{1}$.
    % But then for every $x^{*}\in (X^{*})_+$,
    % $x^{*}(0,1)=\mathbb{1}(x^{*})=\|x^{*}\|$, which by
    % \cref{lem:approx_id_bidual} implies that $(0,1)$ is an order unit.
    % This is a contradiction.
\end{example}

A precise formulation of the second consequence requires working over the
complex field, and therefore requires the introduction of complex Banach
lattice algebras. Let $A$ be a Banach lattice algebra. Applying the procedure for the
complexification of Banach lattices (see \cite[Section
3.2]{abramovich_aliprantis2002}) one
obtains the complex Banach lattice $A_\C$ with norm $\|z\|_\C=\| |z|
\|$. Together with the usual product
\[
    (x_1+ix_2)(y_1+iy_2)=(x_1y_1-x_2y_2)+i(x_1y_2+x_2y_1)\quad\text{where
    }x_1,x_2,y_1,y_2 \in A,
\]
the space $A_\C$ becomes a complex algebra, which has an identity precisely
when $A$ has one. This product is compatible
with the complex lattice structure in the sense that $|z_1 z_2|\le
|z_1| |z_2|$, for all $z_1,z_2 \in A_\C$. This fact is not trivial,
and a proof may be found in \cite{huijsmans1985}. The submultiplicativity of the
norm on $A$ immediately implies that $\|z_1 z_2\|_\C\le \|z_1\|_\C
\|z_2\|_\C$ for $z_1,z_2 \in A_\C$. Hence $A_\C$ is also a complex
Banach algebra with norm $\|{\cdot }\|_\C$. The space $A_\C$ with its
structures of complex Banach lattice and complex Banach algebra is
called a \emph{complex Banach lattice algebra}. When $A$ has an
identity, we
emphasize it by saying that $A_\C$ is a \emph{complex Banach
lattice algebra with identity}.

For example, if $K$ is a compact Hausdorff space, the
complexification of $C(K)$ may be identified with $C(K,\C)$, the space
of continuous complex-valued functions on~$K$. Similarly, the
complexification $A_\C$ of any closed sublattice-algebra $A$ of $C(K)$
(i.e., of any AM-algebra with approximate unit) will be a subalgebra
of $C(K,\C)$. It is clear that $A_\C$, being the complexification of
$A$, is closed under involution. It follows that it is a
C$^{*}$-algebra, in which the cone of self-adjoint elements with positive spectrum
(i.e., the positive elements in the sense of C$^{*}$-algebras) is
precisely $A_+$.  Conversely, by a result of S.\ Sherman
\cite{sherman1951}, any C$^{*}$-algebra $A_\C$
in which the cone of positive elements (in the C$^{*}$-algebra sense)
forms a lattice, must be commutative, hence a closed
subalgebra of $C(K,\C)$. Moreover, $A_\C$ will be the complexification
of $A\cap C(K)$, which is a closed sublattice-algebra of
$C(K)$, hence an AM-algebra with approximate unit. This discussion is
summarized in the final corollary.

\begin{cor}
    Let $A$ be a Banach lattice algebra. Its complexification $A_\C$
    can be endowed with a C$^*$-algebra structure in such a way that
    $A_+$ is the cone of self-adjoint elements with positive spectrum
    if and only if $A$ is an AM-algebra with approximate unit.
\end{cor}

\section{Banach lattice algebras with unique multiplication}\label{sec:products}
%\section{Discussion and proof of \cref{thm:products}}\label{sec:products}

Recall that a \emph{vector lattice algebra} is a vector lattice
together with a real algebra structure in which the product of positive
elements is positive. A \emph{vector lattice algebra with identity} is
a vector lattice algebra together with a positive algebraic identity.
Recall also from \cref{prop:martignon} in the introduction that, fixed
the pointwise lattice structure and the algebraic identity to be the
constant one function, there is a single product in $C(K)$ that makes
it a vector lattice algebra (namely, the pointwise product). The
goal of this section is to prove the converse of this result.

\products*

\begin{rem}
In contrast, given an Archimedean \falg\ with positive identity, there is a unique \falg\ product having the same identity (see \cite[Theorem 2.58]{aliprantis_burkinshaw2006}). Thus, if we only consider \falg\ products, then every single Archimedean \falg\ with positive identity has uniqueness of the product.
\end{rem}

Before we can proceed to the proof, we need the following properties
of Banach lattice algebras with identity.

\begin{thm}\label{thm:bla_ideal}
    Let $A$ be a Banach lattice algebra with identity $e$, and let
    \[ A_e=\{\, a \in A : |a|\le \lambda e\text{ for some }\lambda >0 \,
    \} \]
    be the (order) ideal generated by the identity element.
    \begin{enumerate}
        \item The space $A_e$ is a Banach lattice
            algebra, lattice and algebra isometric to $C(K)$,
            for a certain compact Hausdorff space $K$.
        \item The ideal $A_e$ is a principal projection band in $A$.
    \end{enumerate}
\end{thm}
\begin{proof}[Sketch of proof.]
    It is not difficult to check that $(A_e,\|{\cdot }\|_e)$, where
    $\|x\|_e=\inf \{\, \lambda >0 : |x|\le \lambda e \, \} $ for $x
    \in A_e$, is an AM-algebra with unit $e$. By
    \cref{thm:generalAMalg} there exists a compact Hausdorff space $K$
    such that $(A_e,\|{\cdot }\|_e)$ is lattice and algebra isometric to
    $C(K)$. One can check that $A_e$ is such that if $x \in A_e$ is
    invertible in $A$, then $x ^{-1}\in A_e$; in particular, the spectral
    radius $r(x)$ is the same computed with respect to $A_e$ or the
    whole $A$. From the identification of $A_e$ with $C(K)$ it is then
    clear that $\|x\|_e=r(x)\le \|x\|$. And since
    $\|x\|\le \|x\|_e \|e\|=\|x\|_e$ it follows $\|x\|=\|x\|_e$ for every $x\in A_e$.

    For the second part, see \cite[Theorems 1 and 2]{huijsmans1988}.
\end{proof}

\begin{proof}[Proof of \cref{thm:products}]
    If $A=A_e$,
    then the result follows from \cref{thm:bla_ideal}. Suppose
    that $A\neq A_e$. We are going to exhibit another product that
    makes $A$ a vector lattice algebra with identity $e$.

    Let $P\colon A\to A$ be the band projection onto $A_e$, and let
    $P^{d}$ be its disjoint complement. Let
    $\alpha ,\beta \in A_e^{*}$ be multiplicative functionals with
    $\alpha (e)=\beta (e)=1$ (if we identify $A_e$ with $C(K)$,
    $\alpha$ and $\beta $ are just evaluations at points of $K$). For $x,y
    \in A$, define their product $\star$ as
    \[
    x\star y=PxPy+\alpha (Px)P^{d}y+P^{d}x\beta (Py).
    \]
    First we have to check that $\star$ is actually a product. It is
    clear from the definition that $\star$ is bilinear. For convenice of the reader we include the details for checking associativity: Given $x,y,z \in A$,
    note that $P(x\star y)=Px Py$ whereas $P^{d}(x\star y)=\alpha
    (Px)P^{d}y+P^{d}x \beta (Py)$. Then
    \begin{align*}
        (x\star y)\star z&=P(x\star y)Pz+\alpha (P(x\star
        y))P^{d}z+P^{d}(x\star y)\beta (Pz)\\
                         &=Px Py Pz+\alpha (Px Py)P^{d}z+\alpha
                         (Px)\beta (Pz)P^{d}y+\beta (Py)\beta
                         (Pz)P^{d}x,
    \end{align*}
    and
    \begin{align*}
        x\star (y\star z)&=PxP(y\star z)+\alpha (Px)P^{d}(y\star
        z)+P^{d}x\beta (P(y\star z))\\
                         &=PxPyPz+\alpha (Px)\alpha (Py)P^{d}z+\alpha
                         (Px)\beta (Pz)P^{d}y+\beta (Py Pz)P^{d}x.
    \end{align*}
    Since $\alpha $ and $\beta $ are multiplicative, it follows that
    $(x\star y)\star z=x\star (y\star z)$. 
    
    Now, it is clear that if $x,y
    \in A_+$ then $x\star y \in A_+$. Also,
    \[
    x\star e=PxPe+P^{d}x\beta (Pe)=x \quad\text{and}\quad e\star x=PePx+\alpha
    (Pe)P^{d}x=x,
    \]
    so $e$ is still the identity for $\star$. We have thus shown that $A$,
    equipped with this product, is a vector lattice algebra with
    identity $e$.

    If $A_e\neq
    \R$ (i.e., if the compact $K$ has more than one point), different
    functionals $\alpha $ and $\beta $ will give different products.
    Also, if there exist $x,y \in A_e^{d}$ such that $xy\neq
    0$, then $x\star y=0$, and the product $\star$ is different from
    the one we started with. Thus, in these cases we have
    at least two different products. 
    
    It remains to check what happens
    when $A_e=\R$ and $xy=0$ for all $x,y \in A_e^{d}$. In this case,
    any $x,y \in A$ can be written as $x=\lambda e+x'$ and $y=\mu
    e+y'$ with $\lambda ,\mu \in \R$, $x',y' \in A_e^{d}$, and then
    \[
    xy=\lambda \mu e+\lambda y'+\mu x'.
    \]
    Let $\phi \in (A_e^{d})^{*}_+$ be a non-zero functional, let $x_0
    \in (A_e^{d})_+$ be a non-zero element, and define the product
    \[
    x\ast y= \lambda \mu e+\lambda y'+\mu x'+\phi (x')\phi (y')x_0.
    \]
    This is clearly bilinear. Let $z=\nu e+z'$, with $\nu \in \R$ and
    $z' \in A_e^{d}$, then
    \begin{align*}
        (x\ast y)\ast z&=\lambda \mu \nu e+\lambda \mu z'+\nu
        (\lambda y'+\mu x'+\phi (x')\phi (y')x_0)\\&+\phi (\lambda y'+\mu
        x'+\phi (x')\phi (y')x_0)\phi (z')x_0\\
                       &=\lambda \mu \nu e+\lambda \nu z'+\nu \lambda
        y'+\nu \mu x'\\&+\big[ \nu \phi (x')\phi
                           (y')+\lambda \phi (y')\phi (z')+\mu \phi
                           (x')\phi (z')+\phi (x')\phi (y')\phi
                       (z')\phi (x_0)\big]x_0\\
                       &=x\ast (y\ast z),
    \end{align*}
    and this shows that it is associative. If $x \in A_+$, then
    $\lambda \ge 0$ and $x' \in A_+$, so the product
    of positive elements is
    positive. It is straightforward that $e$ is the identity for this
    product; hence $A$ with $\ast$ is a vector lattice algebra with
    identity $e$. We have chosen $\phi $ to be a non-zero functional; pick
    $x \in A_e^{d}$ such that $\phi (x)\neq 0$. Then $x\ast x=\phi
    (x)^2x_0\neq 0$, whereas $x^2=0$ in the original product of $A$.
    We have again exhibited two different products with the desired
    properties, and this finishes the proof.
\end{proof}

\begin{rem}
    In our definition of Banach lattice algebras with identity we assumed
    that the identity had norm 1, following
    \cite{wickstead2017_questions}. However, it is also common to only
    ask for the
    identity to be positive. This condition is weaker (see
    \cite{bernau_huijsmans1990_unit}). If we only assume positivity of the
    identity element in \cref{thm:bla_ideal}, then $A_e$ is lattice and
    algebra isomorphic to $C(K)$ (but not isometric in general), and
    similarly in \cref{thm:products} we can only conclude that the
    Banach lattice algebra is isomorphic to $C(K)$. The results of
    \cref{sec:AMalg} are not affected by this change.
\end{rem}

\section*{Acknowledgements}

Research supported by grants PID2020-116398GB-I00 and CEX2019-000904-S funded by  MICIU/AEI/10.13039/501100011033. Research of D.~Muñoz-Lahoz supported by an FPI--UAM 2023 contract funded by Universidad Autónoma de Madrid. Research of P.~Tradacete also supported by a 2022 Leonardo Grant for Researchers and Cultural Creators, BBVA Foundation.

\bibliographystyle{amsplain}
\bibliography{library}

\providecommand{\bysame}{\leavevmode\hbox to3em{\hrulefill}\thinspace}
\providecommand{\MR}{\relax\ifhmode\unskip\space\fi MR }
\providecommand{\MRhref}[2]{%
  \href{http://www.ams.org/mathscinet-getitem?mr=#1}{#2}
}
\providecommand{\href}[2]{#2}
\begin{thebibliography}{10}

\bibitem{abramovich_aliprantis2002}
Y.~A. Abramovich and C.~D. Aliprantis, \emph{An invitation to operator theory}, Grad. Stud. Math., vol.~50, American Mathematical Society (AMS), 2002.

\bibitem{alekhno2012}
E.~A. Alekhno, \emph{The irreducibility in ordered {Banach} algebras}, Positivity \textbf{16} (2012), no.~1, 143--176.

\bibitem{alekhno2018}
\bysame, \emph{On the peripheral spectrum of positive elements}, Positivity \textbf{22} (2018), no.~4, 931--968.

\bibitem{aliprantis_burkinshaw2006}
C.~D. Aliprantis and O.~Burkinshaw, \emph{Positive operators}, reprint of the 1985 original ed., Springer, 2006.

\bibitem{bernau_huijsmans1990_unit}
S.~J. Bernau and C.~B. Huijsmans, \emph{On the positivity of the unit element in a normed lattice ordered algebra}, Studia Mathematica \textbf{97} (1990), no.~2, 143--149.

\bibitem{blanco2023}
A.~Blanco, \emph{On the automatic regularity of derivations from {Riesz} subalgebras of {{\(\mathcal{L}^r (X)\)}}}, Indag. Math., New Ser. \textbf{34} (2023), no.~1, 143--167.

\bibitem{buskes_page2005}
G.~Buskes and R.~Page, \emph{A positive note on a counterexample by {Arens}}, Quaestiones Mathematicae \textbf{28} (2005), no.~1, 117--121.

\bibitem{davidson1996}
K.~R. Davidson, \emph{C$^*$-algebras by example}, Fields Institute Monographs, vol.~6, American Mathematical Society, 1996.

\bibitem{de_jeu2021}
M.~de~Jeu, \emph{Free vector lattices and free vector lattice algebras}, Positivity and its applications. Proceedings from the conference Positivity X, Pretoria, South Africa, July 8–-12, 2019 (2021), 103--139.

\bibitem{grobler1988}
J.~J. Grobler, \emph{The zero-two law in {Banach} lattice algebras}, Isr. J. Math. \textbf{64} (1988), no.~1, 32--38.

\bibitem{huijsmans1985}
C.~B. Huijsmans, \emph{An inequality in complex {Riesz} algebras}, Studia Scientiarum Mathematicarum Hungarica \textbf{20} (1985), 29--32.

\bibitem{huijsmans1988}
\bysame, \emph{Elements with unit spectrum in a {Banach} lattice algebra}, Indagationes Mathematicae \textbf{50} (1988), no.~1, 43--51.

\bibitem{huijsmans1995}
\bysame, \emph{A characterization of complex lattice homomorphisms on {Banach} lattice algebras}, Quaestiones Mathematicae \textbf{18} (1995), no.~1-3, 131--140.

\bibitem{huijsmans_depagter1984}
C.~B. Huijsmans and B.~de~Pagter, \emph{The order bidual of lattice ordered algebras}, Journal of Functional Analysis \textbf{59} (1984), 41--64.

\bibitem{jaber2020}
J.~Jaber, \emph{The {Fremlin} projective tensor product of {Banach} lattice algebras}, Journal of Mathematical Analysis and Applications \textbf{488} (2020), no.~2, 9.

\bibitem{kakutani1941}
S.~Kakutani, \emph{Concrete representation of abstract ({M})-spaces. ({A} characterization of the space of continuous functions.)}, Annals of Mathematics. Second Series \textbf{42} (1941), 994--1024.

\bibitem{lindenstrauss_tzafriri1979}
J.~Lindenstrauss and L.~Tzafriri, \emph{Classical banach spaces {II}. function spaces}, Ergeb. Math. Grenzgeb., vol.~97, Springer-Verlag, Berlin, 1979.

\bibitem{martignon1980}
L.~Martignon, \emph{Banach $f\!$-algebras and {Banach} lattice algebras with unit}, Boletim da Sociedade Brasileira de Matemática \textbf{11} (1980), no.~1, 11--17.

\bibitem{palmer1994}
T.~W. Palmer, \emph{Banach algebras and the general theory of $^*$-algebras. volume i: algebras and {Banach} algebras}, Encycl. {Math}. {Appl}., vol.~49, Cambridge University Press, Cambridge, 1994.

\bibitem{Render92}
H.~Render, \emph{Lattice structures of ordered {Banach} algebras}, Ill. J. Math. \textbf{36} (1992), no.~2, 238--250.

\bibitem{scheffold1980}
E.~Scheffold, \emph{Über komplexe {Banachverbandsalgebren}}, Journal of Functional Analysis \textbf{37} (1980), 382--400 (German).

\bibitem{scheffold1981}
\bysame, \emph{{FF}-{Banachverbandsalgebren}}, Mathematische Zeitschrift \textbf{177} (1981), 193--205.

\bibitem{scheffold1984}
\bysame, \emph{Banachverbandsalgebren mit einer natürlichen {Wedderburn}-{Zerlegung}}, Mathematische Zeitschrift \textbf{185} (1984), 521--531 (German).

\bibitem{scheffold1988}
\bysame, \emph{{\"U}ber {Banachverbandsalgebren} mit multiplikativen {Zerlegungseigenschaft}. ({Banach} lattice algebras with multiplicative decomposition property)}, Acta Math. Hung. \textbf{52} (1988), no.~3-4, 273--289.

\bibitem{sherman1951}
S.~Sherman, \emph{Order in operator algebras}, American Journal of Mathematics \textbf{73} (1951), 227--232.

\bibitem{wickstead2017_questions}
A.~W. Wickstead, \emph{Banach lattice algebras: some questions, but very few answers}, Positivity \textbf{21} (2017), no.~2, 803--815.

\bibitem{wickstead2017_ordered}
\bysame, \emph{Ordered {Banach} algebras and multi-norms: some open problems}, Positivity \textbf{21} (2017), no.~2, 817--823.

\bibitem{wickstead2017_2d}
\bysame, \emph{Two dimensional unital {Riesz} algebras, their representations and norms}, Positivity \textbf{21} (2017), no.~2, 787--801.

\end{thebibliography}

\end{document}